\documentclass{article}

\usepackage{pdfpages}
\usepackage[hidelinks]{hyperref}
\usepackage{xcolor}
\usepackage[utf8]{inputenc}
\usepackage{amsmath}
\usepackage{amsfonts}
\usepackage{amssymb}
\usepackage{amsthm}
\usepackage{array}
\usepackage{marginnote}
\usepackage{enumitem}
\usepackage{graphicx}

\usepackage{fancyhdr}

\usepackage{cleveref}

\makeatletter
\let\th@theorem\th@plain
\makeatother
\newcommand{\defstyle}[1]{\textbf{#1}}
\newcommand{\myprob}[1]{\mathbb P \left[ #1 \right]}

\newcommand{\probCond}[2]{\mathbb P \left[ #1 \left| #2 \right. \right]}

\newcommand{\omid}[1]{\mathbb E \left[ #1 \right]}

\newcommand{\omidCond}[2]{\mathbb E \left[ #1 \left| #2 \right. \right]}

\newcommand{\norm}[1]{\left| #1 \right|}

\newcommand{\identity}[1]{1_{#1}}
\newcommand{\bs}[1]{\boldsymbol{#1}}
\newcommand{\card}[1]{\left|#1\right|}

\newcommand{\del}[1]{}
\newcommand{\unwritten}[1]{}

\newcommand{\cball}[2]{\overline{B}_{#1}(#2)}

\newcommand{\restrict}[2]{{
		\left.\kern-\nulldelimiterspace 
		#1 
		\vphantom{\big|} 
		\right|_{#2} 
}}

\theoremstyle{theorem}
\newtheorem{theorem}{Theorem}[section]
\newtheorem{lemma}[theorem]{Lemma}
\newtheorem{proposition}[theorem]{Proposition}

\theoremstyle{definition}
\newtheorem{definition}[theorem]{Definition}

\newtheorem{example}[theorem]{Example}
\theoremstyle{definition}
\newtheorem{remark}[theorem]{Remark}

\newtheorem{claim}{Claim}
\crefname{model}{Model}{Models}
\crefname{submodel}{Model}{Models}
\crefname{algorithm}{Algorithm}{Algorithms}
\crefname{equation}{}{}
\crefname{problem}{Problem}{Problems}
\crefname{theorem}{Theorem}{Theorems}
\crefname{corollary}{Corollary}{Corollaries}
\crefname{claim}{Claim}{Claims}

\theoremstyle{theorem}

\numberwithin{equation}{section}

\makeatletter
\let\orgdescriptionlabel\descriptionlabel
\renewcommand*{\descriptionlabel}[1]{%
	\let\orglabel\label
	\let\label\@gobble
	\phantomsection
	\edef\@currentlabel{#1}%
	\let\label\orglabel
	\orgdescriptionlabel{#1}%
}

\newlist{stepC}{enumerate}{1}
\setlist[stepC,1]{
	label=\textbf{Step C\arabic*},
	ref=C\arabic*,
	leftmargin=* 
}
\crefname{stepCi}{Step}{Steps}
\Crefname{stepCi}{Step}{Steps}

\newlist{stepCC}{enumerate}{1}
\setlist[stepCC,1]{
	label=\textbf{Step C\arabic{stepCi}.\arabic*},
	ref=C\arabic{stepCi}.\arabic*,
	leftmargin=* 
}
\crefname{stepCCi}{Step}{Steps}
\Crefname{stepCCi}{Step}{Steps}

\newlist{stepL}{enumerate}{1}
\setlist[stepL,1]{
	label=\textbf{Step L\arabic*},
	ref=L\arabic*,
	leftmargin=* 
}
\crefname{stepLi}{Step}{Steps}
\Crefname{stepLi}{Step}{Steps}

\newlist{stepLL}{enumerate}{1}
\setlist[stepLL,1]{
	label=\textbf{Step L\arabic{stepLi}.\arabic*},
	ref=L\arabic{stepLi}.\arabic*,
	leftmargin=* 
}
\crefname{stepLLi}{Step}{Steps}
\Crefname{stepLLi}{Step}{Steps}

\makeatother

\begin{document}

\title{On Stopping Sets and Continuous-Time Decision Trees}

\author{Ali Khezeli \footnote{School of Mathematics, Institute for Research in Fundamental Sciences, alikhezeli@ipm.ir}}

\maketitle

\begin{abstract}
We study stopping sets and continuous-time decision trees (CTDTs) for Poisson point processes. We prove that every $\mu$-continuous CTDT has an explicit exploration representation: between successive discoveries, the explored set grows according to a deterministic rule depending on the points already observed. If $\mu$ is diffuse, we prove that the points are revealed one by one almost surely; this need not be imposed as a separate assumption. We also allow atoms of $\mu$, in which case all points at one location are revealed together. For diffuse $\mu$, we also prove that a locally bounded stopping set that reveals at most one Poisson point is attainable by a $\mu$-continuous exploration. More generally, an increasing sequence of locally bounded stopping sets, each revealing at most one additional point, can be attained by a single exploration viewed at successive stopping times. Finally, we prove the spatial Markov property of the Poisson point process without assuming that the stopping set contains finitely many points or admits a finite approximation. The proof is based on a general criterion for characterizing conditionally Poisson point processes using moment measures, and uses the multivariate Mecke formula on the complement of the stopping set. 
\end{abstract}


\section{Introduction}

A stopping set is a spatial analogue of a stopping time. Informally, it is a random region that is determined by the configuration observed inside it: changing the points outside the region does not change the region itself. 
Stopping sets have been originally defined through set-indexed filtrations and have had important applications in stochastic geometry and randomized algorithms. For instance, \cite{Zu99stopping} uses this notion and proves Gamma-type results for the Poisson point process. Also, \cite{Zu06strongMarkov} proves the spatial Markov property for compact stopping sets. In~\cite{LaPeYo23stopping}, Last, Peccati and Yogeshwaran define stopping sets more generally on a measurable space, without topological assumptions, and introduce continuous-time decision trees. Their framework yields Poisson versions of the OSSS and Schramm--Steif inequalities, with applications to phase transitions and noise sensitivity. We use their definition of stopping sets throughout this paper.

A continuous-time decision tree (CTDT) \cite{LaPeYo23stopping} roughly means a growing family of stopping sets $(Z_t)_{t\geq 0}$ indexed by nonnegative real numbers. See \Cref{def:ctdt}, which extends this notion by allowing atoms. As discussed in~\cite{LaPeYo23stopping}, this notion is inspired by randomized algorithms with Boolean inputs, and $Z_t$ is analogous to the region explored by the algorithm at time $t$. However, not all stopping sets and CTDTs are obtained by such an algorithm (even in the Bernoulli case). 

In this paper, rather than discussing applications, we focus on further developing the theory of stopping sets and CTDTs. We introduce the notion of \defstyle{exploration processes}, which is another natural continuum analogue of randomized algorithms with Poisson inputs, and does not rely stopping sets. 
Roughly speaking, given a Poisson point process (PPP) $\Phi$ on the underlying space $\mathbb X$ as the input, the algorithm explores the space incrementally and deterministically until a point of $\Phi$ is discovered. The exploration rule is then updated based on the enlarged observed configuration. 
This notion is closely related to \textit{algorithmically local sets} in~\cite{ScSh13}, with the difference that the latter uses the Gaussian free field as the input, rather than the Poisson point process.

Our aim is to relate stopping sets and CTDTs to exploration processes. It is easy to see that an exploration process is a CTDT (\Cref{prop:ctdt}). The first main result is the converse:

\begin{theorem}[Exploration Representation of CTDTs]
	\label{thm:ctdt}
	Let $\Phi$ be a Poisson point process on $\mathbb X$ with intensity measure $\mu$. For every $\mu$-continuous CTDT $(Z_t)_{t\geq 0}$, there exists an exploration process (\Cref{def:explorationnew}) $(\tilde Z_t)_t$ such that, almost surely, $\forall t\geq 0: \tilde Z_t(\Phi)=Z_t(\Phi)$. 
	In addition, the (multi-) points of $\Phi$ are revealed one by one a.s. 
\end{theorem}

See \Cref{def:ctdt} for $\mu$-continuity, generalized to allow atoms.
A main challenge in this theorem is proving the one-by-one revealing property, which is proved in \Cref{lem:ctdt}. When $\mu$ is diffuse, this means that, almost surely,
\[
    \Phi\bigl(Z_t(\Phi)\setminus Z_{t-}(\Phi)\bigr)\leq 1
    \qquad\text{for all }t\geq 0.
\]
The assertion holds simultaneously at all times, including random discovery times. 
This was assumed in \cite[assumption~(3.5)]{LaPeYo23stopping}, and now, \Cref{lem:ctdt} shows that it automatically follows from $\mu$-continuity. In the non-diffuse case, one location may carry several points, and revealing that location reveals its whole multiplicity. The corresponding assertion is that at most one new location (which we also call a \textit{multi-point}) is revealed at a time. 

The one-by-one property is important since, if more than one point appear simultaneously, then the process might not be attainable by an exploration (see \Cref{def:attain} for the formal definition of attainability). Indeed, \cite[Example~2]{LaPeYo23stopping} constructs a stopping set which cannot be attained by any $\mu$-continuous CTDT (and hence, by any exploration process). We modify the latter and construct a non-attainable stopping set with at most two revealed Poisson points (\Cref{ex:counterexample}).

In the next main result, we study which stopping sets are attainable by exploration (notice that no CTDT is assumed here, and hence, \Cref{thm:ctdt} does not apply). According to the previous paragraph, it is natural to restrict attention to the case where at most one Poisson point is revealed. We prove that this condition guarantees attainability under the following generic assumptions (here, $\mathcal X_0$ is analogous to bounded measurable sets; see \Cref{sec:markov}):



\begin{theorem}[Attainable Stopping Sets]
	\label{thm:attain}
	If $\mu$ is diffuse and $Z$ is a stopping set such that $Z\in \mathcal X_0$ (see \Cref{rem:empty} for a weaker assumption) and $\card{\Phi\cap Z(\Phi)}\leq 1$ a.s., then $Z$ is attainable by a $\mu$-continuous exploration process (and hence, by a $\mu$-continuous CTDT) stopped at the first time a point is discovered. The claim also holds for countable underlying spaces.
		%
\end{theorem}

The proof constructs an almost-everywhere total order from the choices made by the stopping set, and uses this order to construct a deterministic growing family.
It should be noted that, in the countable case, the stopping set may not be attained by a discrete-time exploration (since a well-ordered countable set is not necessarily of the same type as natural numbers; see \Cref{ex:countable}).

More generally, the next theorem proves that an increasing sequence of stopping sets which reveal the points one by one (compare with \Cref{thm:ctdt}) is attained by a single $\mu$-continuous exploration viewed at a sequence of stopping times (see \Cref{def:attain} for the formal definition): 

\begin{theorem}
	\label{thm:attain2}
	If $\mu$ is diffuse and $Z_0\subseteq Z_1\subseteq\cdots$ are stopping sets in $\mathcal X_0$ (see \Cref{rem:empty} for a weaker assumption) such that $Z_0=\emptyset$ and $\forall i\geq 1: \card{\Phi\cap (Z_i(\Phi)\setminus Z_{i-1}(\Phi))}\leq 1$ a.s., then this sequence is attainable by a $\mu$-continuous exploration process (and hence, by a $\mu$-continuous CTDT). The claim also holds for countable underlying spaces.
\end{theorem}

In \Cref{thm:markov} below, which is independent from the previous results, we prove the general version of the \textit{spatial Markov property} of the Poisson point process. Heuristically, one expects that after observing a stopping set, the unobserved part is still Poisson with the original intensity measure restricted to the complement. 
This has been proved in~\cite[Theorem~A.3]{LaPeYo23stopping} for stopping sets containing finitely many Poisson points, and~\cite[Theorem~A.3]{LaPeYo23stopping} treats increasing limits of such stopping sets. Related results for Poisson hulls are proved in~\cite[Theorem~3.2 and Proposition~3.3]{LaMo25hulls}. 
The first version of~\cite[Theorem~2.2]{LaMo26efron} proves the case of pointwise limits of finite stopping sets. 
A day before the present paper was published on arXiv, the new version of~\cite{LaMo26efron} appeared, which independently proves the general spatial Markov property for the broader class of \textit{finitely insertion tolerant} random sets \cite[Theorem~2.2 and Corollary~2.3]{LaMo26efron}. We retain our result since it is done independently and it uses a conditional Poisson characterization (\Cref{prop:conditionalPoisson}). Also, \cite[Theorem~B.2]{BaKh26indistinguishability} prove another version of the spatial Markov property for marked Poisson point processes, where first a subset of unmarked points are revealed, and then the marks of a subset of points are revealed. This result is extended in \Cref{thm:markov2}.

\begin{theorem}[Spatial Markov Property]
	\label{thm:markov}
	Let $\mu$ be a locally finite measure on $\mathbb X$, $\Phi$ be a Poisson point process with intensity measure $\mu$ and $Z$ be a stopping set. Then, conditionally on $\restrict{\Phi}{Z(\Phi)}$, the rest of $\Phi$ is a Poisson point process with intensity measure $\restrict{\mu}{\mathbb X\setminus Z(\Phi)}$. 
\end{theorem}
The proof computes the conditional factorial moment measures of the unobserved process. To deduce the spatial Markov property, a general criterion (\Cref{prop:conditionalPoisson}) is provided for verifying if some point process is conditionally Poisson (conditionally on some random object). The main idea is to apply the multivariate Mecke formula outside $Z(\Phi)$, where adding finitely many points leaves the observed configuration unchanged. This is different from the approach of~\cite{LaPeYo23stopping} (see below), which uses Mecke's equation inside the stopping set to condition on the observed points, and requires a finiteness assumption.

The paper is organized as follows. Section~\ref{sec:markov} proves the spatial Markov property (\Cref{thm:markov}). Section~\ref{sec:ctdt} proves the exploration representation and the one-by-one discovery property (\Cref{thm:ctdt}). Section~\ref{sec:attain} proves the attainability results (\Cref{thm:attain,thm:attain2}). Some problems and possible extensions of the results are discussed in \Cref{sec:problems}. Point configurations are regarded as counting measures, so that multiplicities are retained. When the configuration is simple, we also use set notation for its support; e.g., we use $\Phi\cap A$ for $\restrict{\Phi}{A}$ and $\card{\Phi\cap A}$ for $\Phi(A)$. 

\section{On the Spatial Markov Property of the Poisson Point Process}
\label{sec:markov}

In this section, we first introduce the measurable setting and recall the stopping-set definition. We then prove the spatial Markov property (\Cref{thm:markov}) by identifying the conditional distribution on the complement. 

We use the point-process conventions of~\cite{bookLast17lecturesPP}. Let $(\mathbb X,\mathcal X)$ be a Borel space. Let $(W_n)_{n\geq 1}$ be a sequence of measurable sets such that $W_n\uparrow\mathbb X$ and let $\mathcal X_0:=\{A\in\mathcal X: \exists n: A\subseteq W_n\}$. A measure $\nu$ on $(\mathbb X,\mathcal X)$ is \defstyle{locally-finite} if $\forall n: \nu(W_n)<\infty$. In what follows, by the \defstyle{metric case} we mean the case where $\mathbb X$ is a boundedly-compact metric space and $W_n$ is the ball of radius $n$ with a fixed center.

Fix a compact metric space $\Xi$, called the \textit{mark space}, and let $\bs N$ be the set of integer-valued locally-finite measures on $\mathbb X\times\Xi$ (the latter is equipped with the product of $\mathcal X$ and the Borel sigma-field on $\Xi$). In the metric case, such a measure is equivalent to a marked discrete multi-set of $\mathbb X$ with mark space $\Xi$. Let $\mathcal N$ be the sigma-field on $\bs N$ generated by the maps $\varphi\mapsto \varphi(B)$ for all measurable $B\subseteq\mathbb X\times\Xi$. 
A \defstyle{marked point process} is a random element $\Phi$ of $\bs N$. The \defstyle{intensity measure} of $\Phi$ is the measure $\omid{\Phi}$ on $\mathbb X\times\Xi$ defined by $A\mapsto \omid{\Phi(A)}$. If $\nu$ is a locally-finite measure on $\mathbb X$, the (unmarked) \defstyle{Poisson point process} with intensity measure $\nu$ is the unique (in distribution) point process $\Phi$ such that: (1) For every $C\in\mathcal X$, $\Phi(C)$ is a Poisson random variable with mean $\nu(C)$ and, (2) for every collection of disjoint sets $C_1,\ldots,C_k\in \mathcal X$, where $k\geq 2$, the random variables $\Phi(C_1),\ldots,\Phi(C_k)$ are independent. A marked Poisson point process is defined similarly on $\mathbb X\times\Xi$.
In the marked setting, evaluations on the product space are understood to range over $\mathcal X\otimes\mathcal B(\Xi)$; in particular, these evaluations generate $\mathcal N$. For a ground-space set $A\in\mathcal X$, we use $\varphi(A)$ as shorthand for $\varphi(A\times\Xi)$, and local finiteness means finiteness on each $W_n\times\Xi$. 

We will later need the \textit{factorial measure} $\Phi^{(n)}$, the \textit{factorial moment measure} $\omid{\Phi^{(n)}}$; see~\cite{bookLast17lecturesPP}.
For completeness, if $\varphi=\sum_{j\in J}\delta_{x_j}$, its $n$-th factorial measure is
\[
    \varphi^{(n)}
    :=\sum_{j_1,\ldots,j_n\in J}^{\neq}
      \delta_{(x_{j_1},\ldots,x_{j_n})},\qquad n\geq 1.
\]
The sum is over ordered tuples of distinct indices; the locations need not be distinct when $\varphi$ has multiplicities. The $n$-th factorial moment measure of a point process $\Phi$ is the measure $\alpha_n(A):=\omid{\Phi^{(n)}(A)}$. 
For a PPP with intensity measure $\mu$, one has $\alpha_n=\mu^{\otimes n}$. We will also need the \textit{multivariate Mecke formula},  see \cite[Theorem~4.4]{bookLast17lecturesPP}.

\subsection{The Unmarked Case}

In this subsection, we restrict attention to the unmarked case, or equivalently, when $\Xi$ has only one point. In this case, the following definition is given in~\cite[Appendix~A]{LaPeYo23stopping}.

\begin{definition}[Stopping Sets~\cite{LaPeYo23stopping}]
	\label{def:stopping}
	A \defstyle{stopping set} is a map $Z:\bs N\to\mathcal X$ such that:
	\begin{enumerate}[label=(\roman*)]
		\item \label{def:stopping-equiv} [Graph-measurability] The function $(x,\varphi)\mapsto \identity{Z(\varphi)}(x)$ is a measurable function on $\mathbb X\times\bs N$. 
		\item \label{def:stopping-stopping} [Stopping Property] If $\varphi,\varphi'\in\bs N$ satisfy $\restrict{\varphi'}{Z(\varphi)}=\restrict{\varphi}{Z(\varphi)}$, then $Z(\varphi')=Z(\varphi)$.
	\end{enumerate}
\end{definition}

The stopping property implies $Z(\varphi)=Z(\restrict{\varphi}{Z(\varphi)})$. Thus, the observed configuration determines the stopping set itself. The following theorem identifies the conditional law of the remaining configuration.


We need the following characterization of conditionally Poisson point processes. In this proposition, a \defstyle{locally-finite kernel} is a function $(y,A)\mapsto \nu(y,A)$, defined for $y\in \mathbb Y$ and $A\in\mathcal X$, such that $\nu(\cdot,A)$ is measurable and $\nu(y,\cdot)$ is a locally-finite measure.

\begin{proposition}[Conditionally Poisson Point Processes]
	\label{prop:conditionalPoisson}
	Let $\mathbb Y$ be another measurable space, $\Phi$ be a point process in $\mathbb X$ and $Y$ be a random element of $\mathbb Y$, defined on the same probability space. Let $\nu$ be a locally-finite kernel from $\mathbb Y$ to $\mathbb X$, defined above. Consider the following assertion:
	\begin{equation}
		\label{prop:conditionalPoisson-1}
		\parbox{0.7\textwidth}{
			Conditionally on $Y$, $\Phi$ is a Poisson point process with intensity measure $\nu(Y,\cdot)$; i.e., the distribution of $(\Phi,Y)$ is obtained by constructing $Y$ and then constructing a Poisson point process with intensity measure $\nu(Y,\cdot)$. 
		}
	\end{equation}
	One has:
	\begin{enumerate}[label=(\roman*)]
		\item \label{prop:conditionalPoisson-simple} 
		If $\Phi$ is simple and $\nu(Y,\cdot)$ is diffuse for almost every sample of $Y$, then \Cref{prop:conditionalPoisson-1} is equivalent to the condition that for all measurable subsets $A_1\subseteq \mathbb X$ and $A_2\subseteq \mathbb Y$, one has
		\begin{equation}
			\label{prop:conditionalPoisson-i}
			\myprob{\Phi(A_1)=0, Y\in A_2} = \omid{e^{-\nu(Y,A_1)}\identity{A_2}(Y)}
		\end{equation} 
		\item\label{prop:conditionalPoisson-moment} 
		Even in the non-simple case, \Cref{prop:conditionalPoisson-1} is equivalent to the condition that, 
		for all $n\geq 1$ and measurable $C_1\subseteq \mathbb X^n$ and $C_2\subseteq \mathbb Y$, one has
		\begin{equation}
			\label{prop:conditionalPoisson-ii}
			\omid{\Phi^{(n)}(C_1)\identity{C_2}(Y)} = \omid{\left(\nu(Y,\cdot)^{\otimes n}\right)(C_1) \identity{C_2}(Y)},
		\end{equation}
		where $\Phi^{(n)}$ is the $n$-th factorial measure of $\Phi$ (see~(4.4) in~\cite{bookLast17lecturesPP}).
	\end{enumerate}
\end{proposition}

Part~\ref{prop:conditionalPoisson-simple} generalizes R\'enyi's theorem, which is the case where $Y$ is constant (see~\cite[Theorem~6.10]{bookLast17lecturesPP}).

\begin{proof}
	The implications \eqref{prop:conditionalPoisson-1}$\Rightarrow$\eqref{prop:conditionalPoisson-i} and \eqref{prop:conditionalPoisson-1}$\Rightarrow$\eqref{prop:conditionalPoisson-ii} are clear (see~\cite[Corollary~4.10]{bookLast17lecturesPP}). For \ref{prop:conditionalPoisson-simple}$\Rightarrow$\eqref{prop:conditionalPoisson-1}, 
	let $\Psi$ be the Poisson point process with intensity measure $\nu(Y,\cdot)$ (conditionally on $Y$). 
	The assumption implies that $\probCond{\Phi(A_1)=0}{Y\in A_2} = \probCond{\Psi(A_1)=0}{Y\in A_2}$ provided that $\myprob{Y\in A_2}>0$. So, ~\cite[Theorem~6.11]{bookLast17lecturesPP} implies that the conditional distribution of $\Phi$ conditionally on $Y\in A_2$ coincides with that of $\Psi$. Thus, for every bounded measurable $f$, one has $\omid{f(\Phi)\identity{A_2}(Y)} = \omid{f(\Psi)\identity{A_2}(Y)}$. This implies that $(\Phi,Y)\sim (\Psi,Y)$, which is equivalent to~\eqref{prop:conditionalPoisson-1}.

	

	Now, assume~\ref{prop:conditionalPoisson-moment} holds and define $\Psi$ as above. Let $\Phi':=\{(x,Y): x\in \Phi\}$ and $\Psi':=\{(x,Y): x\in \Psi\}$, respecting the multiplicities, which are point processes in $\mathbb X\times \mathbb Y$. By~\ref{prop:conditionalPoisson-moment}, one obtains 
	\begin{eqnarray*}
		\omid{\Phi'^{(n)}(C_1\times C_2^n)} &=& \omid{\Phi^{(n)}(C_1)\identity{C_2}(Y)}\\
		&=& \omid{\left(\nu(Y,\cdot)^{\otimes n}\right)(C_1) \identity{C_2}(Y)} = \omid{\Psi'^{(n)}(C_1\times C_2^n)}.
	\end{eqnarray*}

	Since both $\Phi'$ and $\Psi'$ lie on the random horizontal section $\mathbb X\times\{Y\}$, one obtains that for every measurable sets $A_{2,1},\ldots,A_{2,n}$ in $\mathbb Y$, and for $A_2:=A_{2,1}\times\cdots\times A_{2,n}$,
	\begin{eqnarray*}
		\omid{\Phi'^{(n)}(C_1\times A_2)} &=& \omid{\Phi'^{(n)}(C_1\times (\cap_i A_{2,i})^n)}\\
		 &=& \omid{\Psi'^{(n)}(C_1\times (\cap_i A_{2,i})^n)} = \omid{\Psi'^{(n)}(C_1\times A_2)}.
	\end{eqnarray*}
	So,  $\Phi'$ and $\Psi'$ share the same factorial moment measures. We will use \cite[Proposition~4.12]{bookLast17lecturesPP} to deduce that $\Phi'$ and $\Psi'$ have the same distribution. Fix $i$ and $M$ such that the event $E_{i,M}:=\{\nu(Y,W_i)\leq M\}$ has positive probability (recall that $W_i\uparrow\mathbb X$). Let $\Phi'_i$ and $\Psi'_i$ be the restrictions of $\Phi'$ and $\Psi'$ to $W_i\times \mathbb Y$ respectively. Conditionally on $E_{i,M}$, the point processes $\Phi'_i$ and $\Psi'_i$ have equal $n$-th factorial moment measures which are bounded by $M^n$. So, \cite[Proposition~4.12]{bookLast17lecturesPP} implies that $\Phi'_i$ and $\Psi'_i$, conditionally on $E_{i,M}$, have the same distribution. Since this holds for all $M$, one obtains that $\Phi'_i\sim \Psi'_i$. Since this holds for all $i$, one gets $\Phi'\sim\Psi'$, and~\eqref{prop:conditionalPoisson-1} is proved.
	%
\end{proof}

\begin{remark}
	\label{rem:Poiss-characterization}
	The above proposition can be extended to marked Poisson point processes using the fact that a marked point process on $\mathbb X$ is just a point process in $\mathbb X\times \Xi$, where $\Xi$ is the mark space.
\end{remark}

\begin{proof}[Proof of \Cref{thm:markov}]
	We will use \Cref{prop:conditionalPoisson} for $\restrict{\Phi}{\mathbb X\setminus Z(\Phi)}$, $Y:=\restrict{\Phi}{Z(\Phi)}$ and $\nu(Y,\cdot):=\restrict{\mu}{\mathbb X\setminus Z(\Phi)}=\restrict{\mu}{\mathbb X\setminus Z(Y)}$.
	Let $n\geq 0$, $f=f(\restrict{\Phi}{Z(\Phi)})$ be a nonnegative measurable function on $\bs N$, and $A\subseteq \mathbb X^n$ be measurable. By \Cref{prop:conditionalPoisson}, it is enough to prove
	\begin{equation}
		\label{eq:thm:markov}
		\omid{\left(\restrict{\Phi}{\mathbb X\setminus Z(\Phi)}\right)^{(n)}(A) f(\restrict{\Phi}{Z(\Phi)})} = \omid{\mu^n\left(A\cap (\mathbb X\setminus Z(\Phi))^n\right) f(\restrict{\Phi}{Z(\Phi)})}.
	\end{equation}
	The left hand side is
	\begin{eqnarray*}
		\omid{\int_{\mathbb X^n}f(\restrict{\Phi}{Z(\Phi)}) \identity{\left(A\cap (\mathbb X\setminus Z(\Phi))^n\right)}(x) d\Phi^{(n)}(x)}.
	\end{eqnarray*}
	By the multivariate Mecke formula \cite[Theorem~4.4]{bookLast17lecturesPP}, and denoting $\delta_{x}=\sum_i \delta_{x_i}$ when $x=(x_1,\ldots,x_n)$, this is equal to
	\begin{eqnarray*}
		\omid{\int_{\mathbb X^n}f(\restrict{(\Phi+\delta_x)}{Z(\Phi+\delta_x)}) \identity{\left(A\cap (\mathbb X\setminus Z(\Phi+\delta_x))^n\right)}(x) d\mu^{n}(x)}.
	\end{eqnarray*}
	The definition of stopping sets implies that, whenever $x\in (\mathbb X\setminus Z(\Phi+\delta_x))^n$, one has $Z(\Phi+\delta_x)=Z(\Phi)$, which in turn implies that $\restrict{(\Phi+\delta_x)}{Z(\Phi+\delta_x)} = \restrict{\Phi}{Z(\Phi)}$ (this is the key idea of using multivariate Mecke formula outside the stopping set). The same conclusions similarly hold whenever $x\in (\mathbb X\setminus Z(\Phi))^n$. Therefore, the above equation is equal to
	\begin{eqnarray*}
		\omid{\int_{\mathbb X^n}f(\restrict{\Phi}{Z(\Phi)}) \identity{\left(A\cap (\mathbb X\setminus Z(\Phi))^n\right)}(x) d\mu^{n}(x)}.
	\end{eqnarray*}
	This is equal to the right hand side of \Cref{eq:thm:markov} and the claim is proved.
\end{proof}


\subsection{The Marked Case}

In the marked case, as mentioned, $\Phi$ is a Poisson point process on $\mathbb X\times\Xi$ with a locally-finite intensity measure $\mu$. As in~\cite[Remarks~A.7]{LaPeYo23stopping}, one may define a stopping set as a graph-measurable map $Z:\bs N(\mathbb X\times\Xi)\to\mathcal X$ similarly to \Cref{def:stopping}. It is straightforward to see that $Z\times\Xi$ is a stopping set on $\mathbb X':=\mathbb X\times\Xi$. Hence, \Cref{thm:markov} provides a version of the spatial Markov property for marked Poisson point processes as well (see~\cite[Remarks~A.7--A.8]{LaPeYo23stopping}).

The author takes another approach to the marked case in~\cite[Appendix~B]{BaKh26indistinguishability} (jointly with Francois Baccelli). Therein, a \defstyle{$d$-marked point process} is a point process $\Phi$ equipped with a marking $m:\Phi^d\to\Xi$. Also, a \defstyle{stopping pair of sets} $Z=(Z_1,Z_2)$ is defined in~\cite{BaKh26indistinguishability}, where $Z_1(\varphi)\in \mathcal X$ describes the set of revealed points of $\varphi$, without revealing their marks, and then $Z_2(\varphi)\subseteq Z_1(\varphi)^d$ describes the set of $d$-tuples of $Z_1(\varphi)$ for which the marks are revealed. Under the finiteness assumption, \cite[Theorem~B.2]{BaKh26indistinguishability} proves a spatial Markov property in this setting. By combining the proof methods of \Cref{thm:markov} and \cite[Theorem~B.2]{BaKh26indistinguishability}, one can prove the following, whose proof is skipped for brevity:

\begin{theorem}[Spatial Markov Property for the $d$-Marked PPP]
	\label{thm:markov2}
	Let $\Phi$ be a Poisson point process with intensity measure $\mu$ and $\bs m$ be an i.i.d. marking of $\Phi^d$. For every stopping pair of sets $Z=(Z_1,Z_2)$, conditionally on $(\restrict{\Phi}{Z_1},\restrict{\bs m}{Z_2})$, the rest of $(\Phi,\bs m)$ is distributed as follows: $\restrict{\Phi}{\mathbb X\setminus Z_1}$ is distributed as a Poisson point process with intensity measure $\restrict{\mu}{\mathbb X\setminus Z_1}$, and conditionally on $\Phi$, $\restrict{\bs m}{\Phi^d\setminus Z_2}$ is an i.i.d. marking of $\Phi^d\setminus Z_2$ with the same mark distribution.
\end{theorem}

\section{On Continuous-Time Decision Trees}
\label{sec:ctdt}

In this section, we define CTDTs and exploration processes, and prove \Cref{thm:ctdt}.
Fix a locally-finite measure $\mu$ on $\mathbb X$ and let $\hat{\mu}$ be the diffuse part of $\mu$ (i.e., $\mu$ minus its atoms). The following definition generalizes that of~\cite{LaPeYo23stopping} by allowing atoms.

\begin{definition}[CTDT]
	\label{def:ctdt}
	A \defstyle{continuous-time decision tree} (CTDT) is a right-continuous nested family $(Z_t)_{t\geq 0}$ of stopping sets such that $Z_t(\varphi)\in\mathcal X_0$ for all $t$ and all $\varphi\in\bs N$.\footnote{In \cite{LaPeYo23stopping}, it is also assumed that $\mu(Z_0(\Phi))=0$ a.s., but we move (an analogue of) this condition to the definition of $\mu$-continuity.} It is called \defstyle{$\mu$-continuous} if, for all $t\geq 0$ and $\varphi\in\bs N$, $\hat\mu(Z_t(\varphi)\setminus Z_{t-}(\varphi))=0$ and $Z_t(\varphi)\setminus Z_{t-}(\varphi)$ contains at most one atom of $\mu$, where $Z_{t-}:=\cup_{s<t} Z_s$ and $Z_{0-}:=\emptyset$. We also denote $Z_{\infty}:=\cup_t Z_t$ and
	\[
	\text{$Z(t,\varphi):=Z_t(\varphi)$ and $Z_-(t,\varphi):=Z_{t-}(\varphi)$}
	\]		
\end{definition}


We now provide an explicit exploration process without relying on the notion of stopping sets.

\begin{definition}[Exploration]
	\label{def:explorationnew}
	Assume that for every finite multiset $B\subseteq \mathbb X$, a one-parameter family $U_t(B)\in\mathcal X_0$ is defined (for $t\geq 0$) such that $B\subseteq U_t(B)$. Assume that $U_s\subseteq U_t$ for $s\leq t$ and $U_t=\bigcap_{s>t} U_s$. Let $U_{t-}:=\bigcup_{s<t} U_s$, where $U_{0-}:=\emptyset$. Assume also that the map $(B,x)\mapsto \identity{\{x\in U_t(B)\}}$ is well defined and measurable for every $t$. 
	Fix a measurable subset $C_0\subseteq \mathbb X$.
	Now, for every $\varphi\in\bs N$,  explore $\varphi$ inductively as follows. Let $\tau_0:=0$, $B_0:=\restrict{\varphi}{C_0}$\footnote{In the proof of \Cref{thm:ctdt}, we will let $C_0$ be the unique atom of $\mu$ in $U_0(\emptyset)$, if there is such an atom.} and $Z_0:=C_0\cup U_0(B_0)$. Given $i\geq 0$, assuming that $B_i, \tau_i$ and $(Z_t)_{t\leq \tau_i}$ are constructed,
	\begin{itemize}
		\item Let $t_i:=\sup\{s: \restrict{\varphi}{U_s(B_i)} = \restrict{\varphi}{U_0(B_i)} \}$ and $\tau_{i+1}:=\tau_i+t_i$.
		\item Let $B_{i+1}:= \restrict{\varphi}{U_{t_i}(B_i)}$. 
		\item For $0<t< t_i$, let $Z_{t+\tau_i}:= Z_{\tau_i} \cup U_t(B_i)$. 
		\item Let $Z_{\tau_{i+1}}:=Z_{\tau_i} \cup U_{t_i}(B_i)\cup U_0(B_{i+1})$.
	\end{itemize}
\end{definition}

Here, $t_i$ is the duration of the $i$-th iteration and $\tau_i$ is its starting time. Since $U_t(B)\in\mathcal X_0$, it is easy to see that $\tau_{i+1}>\tau_i$, and if $\tau_{i+1}<\infty$, then $\restrict{\varphi}{U_{t_i}(B_i)}$ is strictly larger than $\restrict{\varphi}{U_0(B_i)}$; i.e., at least one new point is discovered. The algorithm may have a final stage of infinite duration. The sequence of discovery times may also have a finite limit, denoted by $\tau_\infty:=\lim_i\tau_i$. At the end, the points of $\varphi$ in the explored region $Z_{\infty}:=\cup_t Z_t$ (i.e., $\cup_i B_i$) are revealed.

\begin{remark}
	Note that $U_0(B_{{i+1}})$ may be different from $U_{t_i}(B_i)$ and $B'_{i+1}:=\restrict{\varphi}{U_0(B_{{i+1}})}$ may contain more or less points than $\restrict{\varphi}{U_{t_i}(B_i)}=B_{i+1}$ (this might happen if, e.g., $\varphi$ intersects with $U_0(\emptyset)$ outside $B_0$). 
	This case is counter intuitive since, despite the fact that the points of $B'_{i+1}$ are observed at time $\tau_{i+1}$, the exploration continues in the next step using only the points in $B_{i+1}$ until new points are observed. However, we will show in the proof of \Cref{thm:ctdt} that, if the $\mu$-continuity assumption holds, this case does not happen a.s. But in general, the condition $\restrict{\varphi}{U_s(B_i)} = \restrict{\varphi}{U_0(B_i)}$ in the above definition cannot be replaced by $\restrict{\varphi}{U_s(B_i)} = \restrict{\varphi}{B_i}$. 
\end{remark}

Note that we assumed that the family $U_t(B)$ starts from $t=0$ rather than the time when $B$ is actually observed in the final algorithm. This is because the actual start time in the algorithm can be inferred from $B$. 

The following proposition gives one direction of the relation. The converse, under $\mu$-continuity and for Poisson input, is \Cref{thm:ctdt}.
\begin{proposition}
	\label{prop:ctdt}
	If $(Z_t)_{t\geq 0}$ is defined by an exploration process (\Cref{def:explorationnew}), then it is a CTDT on the interval $[0,\tau_{\infty})$.
\end{proposition}
\begin{proof}
	It is clear that $(Z_t)_t$ is right-continuous and nested. So, we only need to prove that $Z_t$ is a stopping set for fixed $t\geq 0$.
	Fix $t$ and $\varphi,\psi\in\bs N$ such that $\restrict{\psi}{Z_t(\varphi)}=\restrict{\varphi}{Z_t(\varphi)}$. Assume $\tau_i(\varphi)\leq t<\tau_{i+1}(\varphi)$. Right-continuity of $U$ implies that the same holds for some $t<t'<\tau_{i+1}(\varphi)$. By the algorithm that defines $(Z_s)_s$, one can show that $(Z_s)_{s<t'}$ is completely determined by the points of $\varphi$ in $Z_{t'}(\varphi)$. This implies that $\forall s<t': Z_s(\psi)=Z_s(\varphi)$, which proves that $Z_t$ is a stopping set.
	%
	%
	
	For graph-measurability, let $f=f(x,\varphi):=\identity{Z_t(\varphi)}(x)$. By the assumption on $U$, one can show that, for each $i\geq 0$, the pair $(\tau_{i+1},f\cdot \identity{\{\tau_{i+1}>t\}})$ is a measurable function of $(x,\varphi,\tau_0,\ldots,\tau_i)$. So, it is also a measurable function of $(x,\varphi)$ by induction. Then, graph-measurability is implied by the fact $f = \lim_i f\cdot \identity{\{\tau_{i+1}>t\}}$. So, the claim is proved.
\end{proof}

For the converse; i.e., \Cref{thm:ctdt}, we need to consider the discovery times; i.e., the jump times of ${\Phi}(Z_t(\Phi))$. The next lemma guarantees that $Z$ is a stopping set at those jump times. In the statement, we use an alternative definition of stopping times as follows: A random time $\tau$, assumed to be a measurable function of a stochastic process $(\eta_t)_{t\geq 0}$, is a stopping time if $[0,\tau]$ is a stopping set in a sense similarly to \Cref{def:stopping}:  $\tau$ is measurable and, if $t\mapsto \eta'_t$ is another function that coincides with $\eta$ on $[0,\tau(\eta)]$, then $\tau(\eta')=\tau(\eta)$. Similarly to \cite[Proposition~A.1]{BaLa09}, one can show that this definition is equivalent to the classical one (when the canonical filtration of $\eta$ is used).
\begin{lemma}
	If $(Z_t)_{t\geq 0}$ is a CTDT, and $\tau$ is a stopping time for the process $t\mapsto \restrict{\varphi}{Z_t(\varphi)}$, then $Z_{\tau}$ is a stopping set. 
\end{lemma}
This lemma extends the entrance-time construction of~\cite[Lemma~2.2]{LaPeYo23stopping}.
Note that the observed history of $\varphi$ also determines the explored sets. Indeed, for each $t$, the stopping property gives $Z_t(\varphi)=Z_t\bigl(\restrict{\varphi}{Z_t(\varphi)}\bigr)$. 
\begin{proof}
	We denote $\tau$ by $\tau(\varphi)$ since it is determined by $\varphi$. 
	Assume $\varphi,\psi\in\bs N$ satisfy $\restrict{\psi}{Z_{\tau(\varphi)}(\varphi)} = \restrict{\varphi}{Z_{\tau(\varphi)}(\varphi)}$. Monotonicity implies that $\restrict{\psi}{Z_{t}(\varphi)} = \restrict{\varphi}{Z_{t}(\varphi)}$ for all $t\leq \tau(\varphi)$. Since $Z_t$ is a stopping set, one obtains that $Z_t(\psi)=Z_t(\varphi)$ for all $t\leq\tau(\varphi)$ (in particular, $Z_{\tau(\varphi)}(\psi)=Z_{\tau(\varphi)}(\varphi)$). So, $\restrict{\psi}{Z_{t}(\psi)} = \restrict{\varphi}{Z_{t}(\varphi)}$ for all $t\leq \tau(\varphi)$. Hence, the fact that $\tau$ is a stopping time implies that $\tau(\psi)=\tau(\varphi)$. Thus, $Z_{\tau(\psi)}(\psi) = Z_{\tau(\varphi)}(\psi) = Z_{\tau(\varphi)}(\varphi)$. 
	
	It remains to prove graph-measurability. By right-continuity and monotonicity, the function $t\mapsto \restrict{\varphi}{Z_t(\varphi)}$ is c\`adl\`ag. Also, the map $\varphi\mapsto \restrict{\varphi}{Z_t(\varphi)}$ is measurable for fixed $t$. So, \cite[Proposition~3.7.1]{bookEtKu09markov} implies that the map $\varphi\mapsto (\restrict{\varphi}{Z_t(\varphi)})_{t\geq 0}$ is measurable. Hence, the map $\varphi\mapsto \tau(\varphi)$ is also measurable. Since $\min\{\tau,M\}\uparrow\tau$, it is enough to assume that $\tau$ is bounded. By right-continuity,
	\begin{eqnarray*}
		\{x\in Z_{\tau(\varphi)}(\varphi)\} &=& \bigcap_{\tau(\varphi)\leq q\in\mathbb Q}\{x\in Z_q(\varphi)\} \\
		&=& \bigcap_{0\leq q\in\mathbb Q} \big(\{x\in Z_q(\varphi)\} \cup \{q<\tau(\varphi)\}\big).
	\end{eqnarray*}
	It follows that $\{x\in Z_{\tau(\varphi)}(\varphi)\}$ is measurable and graph-measurability is proved.
\end{proof}

The next lemma proves the one-by-one discovery property stated in \Cref{thm:ctdt}. 
Another challenge in \Cref{thm:ctdt} is controlling what happens at a discovery time, when the rule may change as soon as a new point is observed. We control this issue by the one-by-one discovery property, and also by using Part~\ref{lem:ctdt-3} of the next lemma. The latter is a technical result that controls how the stopping sets change at discovery times (see \Cref{rem:ctdt-nonequal} below).


\begin{lemma}[Points Are Revealed One By One]
	\label{lem:ctdt}
	If $\mu$ is diffuse, $(Z_t)_t$ is a $\mu$-continuous CTDT and $\Phi$ is a Poisson point process with intensity measure $\mu$, then
	\begin{enumerate}[label=(\roman*)]
		\item \label{lem:ctdt-1} $Z_0(\Phi)=Z_0(\emptyset)$ and ${\Phi}{(Z_0(\Phi))}=0$ a.s., hence, $Z_0(\Phi)$ is essentially constant.
		\item \label{lem:ctdt-2} Almost surely, $\forall t: \card{\Phi\cap (Z_t(\Phi)\setminus Z_{t-}(\Phi))}\leq 1$.
		\item \label{lem:ctdt-3} Almost surely, 
		\begin{eqnarray*}
			\forall u: \Phi\cap Z_u(\Phi\cap Z_{u-}(\Phi))=\Phi\cap Z_u(\Phi).
		\end{eqnarray*}
	\end{enumerate}
\end{lemma}

\begin{remark}
	\label{rem:lem:ctdt}
	If $\mu$ is not diffuse, the first two parts of the lemma should be modified as follows: For all $t\geq 0$, the set $Z_t(\Phi)\setminus Z_{t-}(\Phi)$ contains at most one multi-point of $\Phi$ a.s. In particular, $Z_0(\Phi)$ also contains at most one multi-point of $\Phi$ (which can only be an atom of $\mu$ in $Z_0(\emptyset)$). The proof is similar and is omitted for brevity (note that, in~\eqref{eq:lem:ctdt-1}, the number of multi-points should be considered without considering their multiplicities).
\end{remark}

\begin{remark}
	\label{rem:ctdt-nonequal}
	In Part~\ref{lem:ctdt-3} of \Cref{lem:ctdt}, one might guess that $Z_u(\Phi\cap Z_{u-}(\Phi))=Z_u(\Phi)$ a.s. This holds for nice CTDTs, but does not hold in general. 
	For example, in the plane, let $\cball{u}{0}$ be the closed ball of radius $u$ around the origin, and let $Z_u(\Phi)$ be the union of $\cball{u}{0}$ and the segments $[x,2x]$ for all $x\in \Phi\cap \cball{u}{0}$.
\end{remark}

\begin{proof}[Proof of \Cref{lem:ctdt}]
	\ref{lem:ctdt-1}.
	One has $\mu(Z_0(\emptyset))=0$, and hence, $\Phi\cap Z_0(\emptyset)=\emptyset$ a.s. On this event, the definition of stopping sets implies that $Z_0(\Phi)=Z_0(\emptyset)$. So, the first claim is proved.
	
	\ref{lem:ctdt-2}. For $\varphi\in\bs N$, let $\tau_i(\varphi)$ be the time of the $i$'th jump of ${\varphi}\cap{Z_t(\varphi)}$; i.e., $\tau_0:=0$ and $\tau_{i+1}(\varphi):=\min\{u>\tau_i(\varphi): \varphi\cap Z_u(\varphi)\neq \varphi\cap Z_{\tau_{i}(\varphi)}(\varphi)\}$ (the facts $\varphi\in \bs N$, $Z_u(\varphi)\in\mathcal X_0$ and right continuity of $(Z_t)_t$ imply that the minimum is attained). For $\tau_i<u<\tau_{i+1}$, one has  ${\varphi}\cap{Z_{u-}(\varphi)}={\varphi}\cap{Z_u(\varphi)}$. So, it is enough to show that, almost surely, 
	\begin{equation}
		\label{eq:lem:ctdt-1}
		\text{$\forall i: \norm{\Phi\cap Z_{\tau_{i+1}(\Phi)}(\Phi)}=1+ \norm{\Phi\cap Z_{\tau_{i}(\Phi)}(\Phi)}$ if $\tau_{i+1}<\infty$.}
	\end{equation}
	Fix $i\geq 0$.
	Let $s:=\tau_i(\Phi)$, $t:=\tau_{i+1}(\Phi)$ and 
	$B_i(\Phi):={\Phi}\cap{Z_s(\Phi)}$. By definition, for $u\in (s,t)$, one has ${\Phi}\cap{Z_u(\Phi)} = B_i(\Phi)$, and hence, $Z_u(\Phi)=Z_u{(B_i(\Phi))}$. In particular, ${\Phi}\cap{Z_{t-}(B_i(\Phi))} = B_i(\Phi)$. We claim that
	\begin{equation}
		\label{eq:lem:ctdt-1.5}
		\Phi\cap Z_t(B_i(\Phi))\neq B_i(\Phi)
	\end{equation}
	if $t<\infty$ (note that $Z_t(B_i(\Phi))$ might be different from $Z_t(\Phi)$ by \Cref{rem:ctdt-nonequal}). Otherwise, the definition of stopping sets would give that $Z_t(\Phi)=Z_t(B_i(\Phi))$, and thus, $\Phi\cap Z_{\tau_{i+1}(\Phi)}(\Phi)=B_i(\Phi)$, a contradiction with the definition of $\tau_{i+1}(\Phi)$. So, we have proved that 
	\[
	\text{$\tau_{i+1}(\Phi)=\tau'_{i+1}(\Phi)$, where }\tau'_{i+1}(\Phi):=\min\{u>\tau_i(\Phi): \Phi\cap Z_u(B_i(\Phi))\neq B_i(\Phi)\}.
	\]
	We now prove that, almost surely,
	\begin{equation}
		\label{eq:lem:ctdt-2}
		\card{\Phi\cap Z_{\tau_{i+1}(\Phi)}(B_i(\Phi))} = \card{B_i(\Phi)}+1 \text{ if $\tau_{i+1}(\Phi)<\infty$.}
	\end{equation}
	Let $W(\Phi):=Z(\tau'_{i+1}(\Phi),B_i(\Phi))\setminus Z_-(\tau'_{i+1}(\Phi),B_i(\Phi))$ if $\tau'_{i+1}<\infty$, and $W(\Phi):=\emptyset$ otherwise. Since $\tau_{i+1}(\Phi)=\tau'_{i+1}(\Phi)$, \eqref{eq:lem:ctdt-2} is equivalent to $\card{\Phi\cap W(\Phi)}=1$ if $\tau'_{i+1}(\Phi)<\infty$, a.s. One has 
	\begin{eqnarray*}
		&&\omid{\card{\Phi\cap W(\Phi)}\cdot\big(\card{\Phi\cap W(\Phi)}-1\big)}\\
		&=& \omid{\sum_{x\in\Phi}\sum_{x\neq y\in \Phi} \identity{\{\{x,y\}\subseteq W(\Phi)\}}}\\
		&=& \omid{\int\int \identity{\{\{x,y\}\subseteq W(\Phi\cup\{x,y\}) \}}d\mu(x)d\mu(y)},
	\end{eqnarray*}
	where we have used the two-variate Mecke formula.
	The reader can verify that, if $\{x,y\}\subseteq W(\Phi\cup\{x,y\})$, then $W(\Phi\cup\{x,y\})=W(\Phi\cup\{x\})$. So, the last equation is at most 
	\begin{eqnarray*}
		&&\omid{\int\int \identity{\{\{x,y\}\subseteq W(\Phi\cup\{x\}) \}}d\mu(x)d\mu(y)}\\
		&=& \omid{\int \mu(W(\Phi\cup\{x\})) \identity{\{x\in W(\Phi\cup\{x\}) \}}d\mu(x)}= 0,
	\end{eqnarray*}
	where we have used the fact that $\mu(W(\Phi\cup\{x\}))=0$, which holds by $\mu$-continuity. So, we have shown that $\omid{\card{\Phi\cap W(\Phi)}\cdot\left(\card{\Phi\cap W(\Phi)}-1\right)}=0$, which implies that $\card{\Phi\cap W(\Phi)}\leq 1$ a.s. Now, \eqref{eq:lem:ctdt-1.5} implies that $\card{\Phi\cap W(\Phi)}=1$ a.s., and \Cref{eq:lem:ctdt-2} is proved.
	
	Since $Z_{\tau_{i+1}(\Phi)}(B_i(\Phi))$ may be different from $B_{i+1}=Z_{\tau_{i+1}(\Phi)}(\Phi)$, we still need some work to deduce~\eqref{eq:lem:ctdt-1} from~\eqref{eq:lem:ctdt-2}.
	Assume $t=\tau_{i+1}(\Phi)<\infty$ and let $\Psi:=B_i(\Phi)$ and $\Psi':=B'_{i+1}(\Phi):=\Phi\cap Z_{t}(\Psi)$. By \Cref{eq:lem:ctdt-2}, we may assume $\Psi'=\Psi\cup\{x\}$ for some $x\in \Phi\setminus \Psi$. The set $Z_t(\Psi')$ might be different from $Z_t(\Psi)$, but since $Z_{t-}(\Psi')=Z_{t-}(\Psi)$, $\mu$-continuity implies that $\mu(Z_t(\Psi')\Delta Z_t(\Psi))=0$. One also has $\Psi\subseteq Z_t(\Psi')$. Similarly to \Cref{eq:lem:ctdt-2}, we will prove that $\left(\Phi\cap Z_t(\Psi')\right)\setminus \Psi'$ is empty a.s. One has
	\begin{eqnarray*}
		\omid{\card{\left(\Phi\cap Z_t(\Psi')\right)\setminus \Psi'}}&=& \omid{\sum_{y\in \Phi} \identity{\{y\in Z_{t}(\Psi')\setminus\Psi'\}}}\\
		&=& \omid{\int \identity{\{y\in Z_{t}(\Psi'(\Phi\cup\{y\}))\setminus\Psi'(\Phi\cup\{y\}) \}}d\mu(y)}.
	\end{eqnarray*}
	In the last term, $t={\tau_{i+1}}$ depends on $\Phi\cup\{y\}$, but the reader can verify that, if the last indicator function is nonzero, then ${t}(\Phi\cup\{y\})={t}(\Phi)$, $\Psi'(\Phi\cup\{y\})=\Psi'(\Phi)$ and $y\in Z_{t}(\Psi')\setminus Z_{{t}-}(\Psi')$. So, the last equation is at most
	\begin{eqnarray*}
		\omid{\int \identity{\{y\in Z_{t}(\Psi'(\Phi))\setminus Z_{{t}-}(\Psi'(\Phi)) \}}d\mu(y)}&=&\omid{\mu(Z_{t}(\Psi')\setminus Z_{{t}-}(\Psi'))}=0.
	\end{eqnarray*}
	This proves that $\left(\Phi\cap Z_t(\Psi')\right)\setminus \Psi'=\emptyset$ a.s. as claimed. Thus, either $\Phi\cap Z_t(\Psi')=\Psi$ or $\Phi\cap Z_t(\Psi')=\Psi'$. In the first case, one has $\Psi\cap Z_t(\Psi')=\Psi'\cap Z_t(\Psi')=\Psi$. So, the definition of stopping sets implies that $Z_t(\Psi)=Z_t(\Psi')$. This contradicts $\Psi'\cap Z_t(\Psi) = \Psi'$. Hence, the second case holds; i.e., $\Phi\cap Z_t(\Psi')=\Psi'$. Again, the definition of stopping sets implies that $Z_t(\Phi)=Z_t(\Psi')$, and hence, $\Phi\cap Z_t(\Phi)=\Psi'$. Now, \eqref{eq:lem:ctdt-2} implies \Cref{eq:lem:ctdt-1} and the proof is completed.
	
	\ref{lem:ctdt-3}.
	Note that only the jump times $u=t=\tau_{i+1}(\Phi)$ are nontrivial here. In this case, $\Phi\cap Z_{t-}(\Phi)=B_i(\Phi)$ in the above notation, and the claim is that $\Phi\cap Z_t(B_i(\Phi))=\Phi\cap Z_t(\Phi)$, which is already proved in the above proof.
\end{proof}

We are now ready to prove \Cref{thm:ctdt}.


\begin{proof}[Proof of \Cref{thm:ctdt}]
	For every finite multiset $B\subseteq \mathbb X$, let $s:=s(B)\in[0,\infty]$ be the minimum time such that $B\subseteq Z_s(B)$. Let $U_t(B):=Z_{t+s}(B)$ if $s<\infty$. In the useless case $s=\infty$, let $\forall t: U_t(B):=B$. It is clear that $U_t(B)$ contains $B$ and is monotone and right-continuous. Also, $\hat\mu(U_0(\emptyset))=0$ and $U_0(\emptyset)$ has at most one atom of $\mu$, namely $x_0$. If $x_0$ exists, construct $\tilde B_i, \tilde\tau_i, \tilde t_i$ and $(\tilde Z_t)_{t\geq 0}$ by the exploration corresponding to $U_t$ starting from revealing $C_0:=\{x_0\}$ (\Cref{def:explorationnew}). Otherwise, do the same starting from $C_0:=\emptyset$. We will prove that $\tilde{Z}_t$ satisfies the claim.
	
	As in the proof of \Cref{lem:ctdt}, let $\tau_i(\Phi)$ be the time of the $i$'th jump of $\restrict{\Phi}{Z_t(\Phi)}$. Let also $\Psi_i:=\restrict{\Phi}{Z(\tau_i(\Phi),\Phi)}$. By induction on $i$, we prove that, almost surely, all $t\in [\tau_i(\Phi),\tau_{i+1}(\Phi))$ satisfy $Z_t(\Phi)=\tilde Z_t$, and also $\tilde \tau_i=\tau_i$ and $\tilde B_i=\Psi_i$. 
	
	The induction basis $i=0$, in the diffuse case, is easily implied by part~\ref{lem:ctdt-1} of \Cref{lem:ctdt}. The general case is as follows.
	If $x_0$ does not exist, then it is clear that $\tilde Z_t(\Phi)=Z_t(\Phi)$ for all $t<\tau_1(\Phi)$. So, assume that $x_0$ exists.	
	By $\mu$-continuity, $C:=\restrict{\Phi}{Z_0(\emptyset)}$ is equal to $\{x_0\}$ with some multiplicity $m_0\geq 0$, and thus, $C=\tilde B_0$.
	Condition on knowing $m_0$.
	If $x_0\not\in Z_0(C)$, then $\restrict{C}{Z_0(C)} =\emptyset$, and hence, $Z_0(C)=Z_0(\emptyset)\ni x_0$, a contradiction. So, $x_0\in Z_0(C)$. Hence, $Z_0(C)$ has no atom of $\mu$ other than $x_0$. 
	Hence, $\mu$-continuity implies that $\restrict{\Phi}{Z_0(C)}=C=\restrict{C}{Z_0(C)}$, and so, $Z_0(\Phi)=Z_0(C)=\tilde Z_0$. Thus, $\restrict{\Phi}{Z_0(\Phi)}=C$.
	Also, for all $t<\tau_1(\Phi)$, one has $\restrict{\Phi}{Z_t(C)}=C=\restrict{C}{Z_t(C)}$. This implies that $Z_t(\Phi)=Z_t(C)=\tilde Z_t$, as desired. So, the induction basis is established.

	
	Now, assuming the claim holds for $1,\ldots,i$, we prove the induction claim for $i+1$. The case $\tau_{i+1}=\infty$ is clear, and hence, assume $\tau_{i+1}<\infty$. In~\eqref{eq:lem:ctdt-1} in the proof of \Cref{lem:ctdt} (see also \Cref{rem:lem:ctdt}), we have shown that $\Psi_{i+1}$ has exactly one (multi-) point more than $\Psi_i$ a.s. This implies that $\tilde\tau_{i+1}(\Phi)=\tau_{i+1}(\Phi)$. Also, it is proved that $Z_{\tau_{i+1}(\Phi)}(\Phi) = Z_{\tau_{i+1}(\Phi)}(\Psi_{i+1})$ and the restriction of $\Phi$ to the latter (which is just $\Psi_{i+1}$) coincides with the restriction of $\Phi$ to $Z_{\tau_{i+1}(\Phi)}(\Psi_{i})$ a.s. This implies that $\tilde B_{i+1} = \Psi_{i+1}$, and that $U_0(\tilde{B}_{i+1}) = Z_{\tau_{i+1}(\Phi)}(\Psi_{i+1}) = Z_{\tau_{i+1}(\Phi)}(\Phi)$. The latter implies that $\restrict{\Phi}{U_0(\tilde{B}_{i+1})} = \tilde{B}_{i+1}$. 
	Also, $U_r(\tilde{B}_{i+1}) = Z_{r+\tau_{i+1}}(\tilde{B}_{i+1})$ by the definition of $U_r$. If $r<\tilde t_{i+1}$, then no new point is discovered up to time $r+\tau_{i+1}$; i.e., $\Phi$ coincides with $\tilde{B}_{i+1}$ on $Z_{r+\tau_{i+1}}(\Phi)$. Since $Z_{r+\tau_{i+1}}$ is a stopping set, one obtains that $Z_{r+\tau_{i+1}}(\Phi) = Z_{r+\tau_{i+1}}(\tilde B_{i+1}) = U_r(\tilde{B}_{i+1})$. Now, the exploration algorithm implies that, for $r<\tilde t_{i+1}$,
	\begin{eqnarray*}
		\tilde Z_{r+\tau_{i+1}} &=& \tilde Z_{\tau_{i+1}-} \cup U_r(\tilde B_{i+1})\\
		&=& Z_{\tau_{i+1}-}(\Phi) \cup Z_{r+\tau_{i+1}}(\Phi) = Z_{r+\tau_{i+1}}(\Phi).
	\end{eqnarray*}
	So, the induction claim is proved and the proof is completed.
\end{proof}



\section{Stopping Sets Attainable by CTDTs}
\label{sec:attain}

We now study whether a given stopping set or a sequence of increasing stopping sets can be obtained by exploration, and prove \Cref{thm:attain,thm:attain2}. 



We slightly modify the definition of attainable stopping sets (\cite{LaPeYo23stopping}) and generalize it to a sequence of stopping sets as follows. Recall that $\Phi$ is a Poisson point process with intensity measure $\mu$.

\begin{definition}
	\label{def:attain}
	For $\mu$ and $\Phi$ as above, a stopping set $Z$ is called \defstyle{attainable by CTDTs} (resp. \defstyle{attainable by exploration}) if there is a $\mu$-continuous CTDT (resp. an exploration process) $(\tilde Z_t)_t$ and a stopping time $\tau$ for the process $t\mapsto \Phi\cap \tilde Z_t(\Phi)$ such that $\mu(\tilde Z_{\tau}(\Phi)\Delta Z(\Phi))=0$ and $\Phi\cap \tilde Z_{\tau}(\Phi) = \Phi\cap Z(\Phi)$ a.s.
	\\
	Similarly, a sequence $Z_1\subseteq Z_2\subseteq\cdots$ of stopping sets is called attainable if there exists a $\mu$-continuous CTDT (resp. an exploration process) $(\tilde Z_t)_t$ and a sequence of stopping times $\tau_1\leq \tau_2\leq\cdots$ as above, such that, for all $i$, $\tilde Z_{\tau_i}(\Phi)$ and $Z_i(\Phi)$ satisfy the same conditions as above. 
\end{definition}

Note that no $\mu$-continuity is assumed for attainability by exploration. But if this $\mu$-continuity holds, then the two definitions are equivalent by \Cref{thm:ctdt}.

Before proving the main results, we provide two examples. \Cref{ex:counterexample} is a modification of \cite[Example~2]{LaPeYo23stopping} and constructs a non-attainable stopping set that reveals at most two Poisson points. 

\begin{example}[A Non-Attainable Stopping Set]
	\label{ex:counterexample}
	We start by a Bernoulli point process $\Phi$ on $\mathbb Z$. We denote sets without brackets and without commas; e.g., $13:=\{1,3\}$. Define $Z(\emptyset):=Z(2):=13$, $Z(3):=Z(13):=23$, $Z(12):=Z(123):=12$ and $Z(1):=Z(23):=123$. For all other $\varphi$, let $Z(\varphi):=Z(\varphi\cap 123)$. One can verify that $Z$ is a stopping set that reveals at most two Bernoulli points. Also, $Z$ cannot be attained by an exploration since the first explored point should belong to $Z(\varphi)$ for all $\varphi$.
	\\
	For the Poisson point process on $\mathbb R$, modify the last definition by replacing each $i\in 123$ with the interval $[i,\tau_i]$, where $\tau_i$ is the first point in $[i,i+1]$ (and with $[i,i+1)$ if there is no such point). One can similarly prove non-attainability using \cite[Proposition~B.1]{LaPeYo23stopping}.
\end{example}

The next example provides a stopping set on a countable set that discovers a single point, but is not attainable by a discrete-time exploration.

\begin{example}[Countable Underlying Space]
	\label{ex:countable}
	Let $(\mathbb X,\prec)$ be an infinite countable well-ordered set not isomorphic to $\mathbb N$ (e.g., $\mathbb N^2$ with the lexicographic ordering). Let $\mu$ be a measure on $\mathbb X$ such that $\forall x: \mu(\{\cdot\prec x\})<\infty$ and $Z(\varphi):=\{x: \forall y\in \varphi: x\preceq y\}$. It can be seen that $Z$ is a stopping set that discovers at most one point but cannot be attained by a discrete-time exploration.
\end{example}

\begin{remark}
	\label{rem:empty}
	The conditions $Z\in\mathcal X_0$ and $Z_i\in\mathcal X_0$ in \Cref{thm:attain,thm:attain2} can be weakened as follows with the same proofs: For every $\varphi\in\bs N$, either $\varphi\cap Z(\varphi)=\emptyset$ or $Z(\varphi)\in\mathcal X_0$ (this allows, e.g., that $Z(\emptyset)\not\in\mathcal X_0$). However, we need to allow the exploration process (which will be constructed on a subset of the time interval $[0,1]$) to exit $\mathcal X_0$ at time 1. In addition, it is enough to assume this only for almost every $k$-point element of $\bs N$ (according to the measure $\mu^{\otimes k}$) for every $k\geq 1$. Another weakening is that either $\varphi\cap Z(\varphi)=\emptyset$ or $\mu(Z(\varphi))<\infty$, but the constructed exploration may not be in $\mathcal X_0$. In this case, we should apply the exploration only to the PPP  or to typical $k$-point subsets, since it might not be well defined for general $\varphi$ (see \Cref{sec:problems}).
\end{remark}

\begin{proof}[Proof of \Cref{thm:attain}]
	Proving measurability of the constructions is omitted for brevity.
	The following claim holds for every stopping set:
	\begin{claim}
		\label{claim:attain:empty}
		For $\varphi\in \bs N$, $\varphi\cap Z(\emptyset)=\emptyset$ if and only if $\varphi\cap Z(\varphi)=\emptyset$.
	\end{claim}
	\begin{proof}
		If $\varphi\cap Z(\emptyset)=\emptyset$, then the fact $\varphi \cap Z(\emptyset) = \emptyset\cap Z(\emptyset)$ implies that $Z(\varphi)=Z(\emptyset)$, and hence, $\varphi\cap Z(\varphi)=\emptyset$. Conversely, if $\varphi\cap Z(\varphi)=\emptyset$, then the fact $\varphi\cap Z(\varphi) = \emptyset\cap Z(\varphi)$ implies that $Z(\varphi)=Z(\emptyset)$, and hence, $\varphi\cap Z(\emptyset)=\emptyset$. This proves the claim.
	\end{proof}
	
	Let $E$ be the event that $\Phi$ intersects $Z(\emptyset)$. By \Cref{claim:attain:empty}, outside $E$, one has $\Phi\cap Z(\Phi)=\emptyset$ and $Z(\Phi)=Z(\emptyset)$. Inside $E$, one almost surely has $\card{\Phi \cap Z(\Phi)}=1$. In this case, if $\Phi\cap Z(\Phi)=\{x\}=\{x\}\cap Z(\Phi)$, then one has $Z(x)=Z(\Phi)\ni x$, and thus, $x\in Z(\emptyset)$.
	If $\mu(Z(\emptyset))=0$, then the last arguments show that $\Phi\cap Z(\Phi)=\emptyset$ a.s. and the claim of the theorem is trivial. So, assume $\mu(Z(\emptyset))>0$.

	In the next claim, $\mu^{\otimes k}$ is regarded as a measure on the set of $k$-point elements of $\bs N$.
	\begin{claim}
		\label{claim:attain:k}
		For every $k\geq 1$ and for $\mu^{\otimes k}$-a.e. $\varphi$, if $\varphi \in E$, then $\card{\varphi\cap Z(\varphi)}=1$ and $Z(\varphi)\in\mathcal X_0$.
	\end{claim}
	\begin{proof}
		Recall that $W_n\uparrow \mathbb X$. Fix $n$ and normalize $\restrict{\mu}{W_n}$ into a probability measure $\nu_n$. Let $\Phi_k$ have distribution $\nu_n^{\otimes k}$. It is enough to prove that $\myprob{A}=0$, where $A$ is the event that $\card{\Phi_k\cap Z(\Phi_k)}>1$ and $Z(\Phi_k)\in \mathcal X_0$ (the last condition is needed in \Cref{rem:empty}). Let $\Phi'_k$ be an independent Poisson point process with intensity measure $\restrict{\mu}{(\mathbb X\setminus W_n)}$, and $\Psi_k:=\Phi_k\cup\Phi'_k$.  If $\myprob{A}>0$, then with positive probability, one has  $\card{\Phi_k\cap Z(\Phi_k)}>1$ and $\Phi'_k\cap Z(\Phi_k)=\emptyset$. In this case, $\Psi_k\cap Z(\Phi_k) = \Phi_k\cap Z(\Phi_k)$, and hence $Z(\Psi_k)=Z(\Phi_k)$. Now, $\card{\Psi_k\cap Z(\Psi_k)}>1$ with positive probability, which contradicts the assumption and the fact that $\Psi_k$ is a Poisson point process conditioned on having $k$ points in $W_n$.
	\end{proof}
	
	\begin{claim}
		\label{claim:attain:<}
		For $\mu^{\otimes 2}$-a.e. $(x,y)$ that intersect $Z(\emptyset)$, either all the conditions $Z(x,y)=Z(x)$, $x\in Z(x)$, $y\not\in Z(x)$ and $\{x,y\}\subseteq Z(y)$ hold, or vice versa.
	\end{claim}
	\begin{proof}
		For almost every such $(x,y)$, the previous claims imply that $\{x,y\}\cap Z(x,y)$ has exactly one point. Assume that it is equal to $\{x\}$. It follows that $Z(x)=Z(x,y)$ and $y\not\in Z(x)$. If $x\not\in Z(y)$, then $\{x,y\}\cap Z(y)=\{y\}\cap Z(y)$, which implies that $Z(y)=Z(x,y)\ni x$, a contradiction. So, $x\in Z(y)$. Also, $y\in Z(y)$ by \Cref{claim:attain:empty}.
	\end{proof}
	
	\begin{claim}
		\label{claim:attain:zero}
		One has $\mu(Z(\Phi)\setminus Z(\emptyset))=0$ a.s.
	\end{claim}
	\begin{proof}
		By the discussion after \Cref{claim:attain:empty}, it is enough to show that, for $\mu$-a.e. $x\in Z(\emptyset)$, one has $\mu(Z(x)\setminus Z(\emptyset))=0$. If this is false, with positive $\mu^{\otimes 2}$ measure, one has $\{x,y\}\cap Z(\emptyset)=\{x\}$, $y\in Z(x)$ and $\{x,y\}$ satisfies the conditions of \Cref{claim:attain:<}. But this is impossible since both of $y\in Z(x)$ and $x\in Z(y)=Z(\emptyset)$ hold (the latter holds by \Cref{claim:attain:empty}).
	\end{proof}

	For $x\neq y\in Z(\emptyset)$, say $x\prec y$ if $\{x,y\}\cap Z(x,y)=\{x\}$. By the proof of the last claim, it implies that $x\in Z(y)$, $y\not\in Z(x)$, and $y\not\prec x$.
	
	\begin{claim}
		\label{claim:attain:transitive}
		For $\mu^{\otimes 3}$-a.e. $(x,y,z)\in Z(\emptyset)^3$, if $x\prec y$ and $y\prec z$, then $x\prec z$.
	\end{claim}
	\begin{proof}
		For simplicity of notation, we denote sets without brackets and commas; e.g., $xyz:=\{x,y,z\}$.
		By \Cref{claim:attain:k}, we may assume that $C:=xyz\cap Z(xyz)$ has exactly one element. If $C=z$, then $xyz\cap Z(xyz)=z\cap Z(xyz)$, and hence, $Z(z)=Z(xyz)\not\ni y$, a contradiction. If $C=y$, then $xyz\cap Z(xyz)= y\cap Z(xyz)$, and hence, $Z(y)=Z(xyz)\not\ni x$, a contradiction. Thus, $C=x$. So, $xyz\cap Z(xyz)= xz\cap Z(xyz)$, and hence, $Z(xz)=Z(xyz)$. Thus, $xz\cap Z(xz)=x$, and so, $x\prec z$.
	\end{proof}
	
	For $x\in Z(\emptyset)$, denote the set $\{y: y\prec x\}$ by $\{\cdot \prec x\}$, and define the set $\{x\prec\cdot\}$ similarly. Let $\mu'$ be any probability measure on $Z(\emptyset)$ that is mutually absolutely continuous with respect to $\restrict{\mu}{Z(\emptyset)}$. Define $f(x):=\mu'(\cdot\prec x)$.
	
	In the countable case, the last claim shows that $\prec$ is a total order on $Z(\emptyset)$. 
	This order determines the desired exploration: explore the space by revealing $V_t:=\{x\in Z(\emptyset): \mu'(\{\cdot\prec x\})\leq t\}$ at time $t$. It is left to the reader to check that this exploration satisfies the claim.
	
	In the diffuse case, the last claim shows that $\prec$ is almost a total order, but we do not know if there exists a full-measure set $A\subseteq Z(\emptyset)$ such that $\prec$ is really a total order on $A$ (after identifying $x$ and $y$ if neither $x\prec y$ nor $y\prec x$). This would simplify the next arguments greatly.

	\begin{claim}
		\label{claim:attain:include}
		For $\mu^{\otimes 2}$-a.e. $(x,y)\in Z(\emptyset)^2$, if $x\prec y$, then $\mu(\{y\prec \cdot \}\setminus\{x \prec \cdot\})=0$ and $\mu(\{\cdot\prec x\}\setminus\{\cdot\prec y\})=0$. In particular, $f(x)\leq f(y)$.
	\end{claim}
	\begin{proof}
		Consider the set $\{x,y,z: x\prec y, y \prec z, x\not\prec z\}$, which has zero measure by the previous claim. The first claim is implied by Fubini's theorem applied to the $z$-sections of this set. The second is implied similarly by considering the $x$-sections and renaming the variables.
	\end{proof}
	
	Let $Y_1,Y_2,\ldots$ be i.i.d. points with distribution $\mu'$. For every $0\leq t\leq 1$, let
	\[
		U_t:=\bigcup_{i: f(Y_i)\leq t} \{\cdot\prec Y_i\},\quad V_t:=\bigcap_{i: f(Y_i)> t} \{\cdot\prec Y_i\},
	\]
	where $U_t=\emptyset$ (resp. $V_t:=Z(\emptyset)$) if there is no $i$ that $f(Y_i)\leq t$ (resp. $f(Y_i)>t$).
	It is clear that, in every realization of $(Y_i)_i$, $U_t$ and $V_t$ are monotone in $t$, and $V_t$ is right-continuous.
	
	\begin{claim}
		\label{claim:attain:=}
		There is an event $E'$ with full probability such that, for every realization of $(Y_i)_i$ in $E'$, one has $\forall t\in [0,1]: \mu'(U_t)=\mu'(V_t)=t$.
	\end{claim}
	\begin{proof}
		We will define $E'$ by countably many conditions which will be added throughout the proof. First, add the condition that $\prec$ is a total order on $(Y_i)_i \cup (x_1,\ldots,x_k)$ for every $k$ and $\mu^{\otimes k}$-a.e. $(x_1,\ldots,x_k)\in Z(\emptyset)^k$. By \Cref{claim:attain:<,claim:attain:transitive}, one can show that this condition is almost surely satisfied. 
		Then, by \Cref{claim:attain:include}, if $f(Y_i)<f(Y_j)$, then $Y_i\prec Y_j$. Also, for almost every $x$, if $f(x)<f(Y_i)$ (resp. $f(x)>f(Y_i)$), then $x\prec Y_i$ (resp. $Y_i\prec x$). In addition, \Cref{claim:attain:include} implies that $\mu'(U_t) = \sup\{f(Y_i): f(Y_i)\leq t \}$. 
		One can similarly show that $\mu'(V_t) = \inf\{f(Y_i): f(Y_i)> t \}$ with the convention $\mu'(V_t):=\mu'(Z(\emptyset))=1$ if there is no such $i$. The previous paragraph shows that $\mu'(U_t\setminus V_t)=0$.
		For every rational $t$, add the condition that $\mu'(V_t)=\mu'(U_t)$. We will show that this condition holds a.s. On this event, monotonicity implies that $\mu'(U_t)=t$ even for irrational $t$.
		
		Fix a rational $t$. We need to prove that $\mu'(V_t\setminus U_t)=0$ a.s. One has 
		\begin{eqnarray*}
			\omid{\mu'(V_t\setminus U_t)} 
			= \myprob{Y\in V_t\setminus U_t},
		\end{eqnarray*}
		where $Y$ is a random point with distribution $\mu'$ independently from $(Y_i)_i$. Let $J_t:=\{i: f(Y_i)\leq t\}$. Almost surely, $f(Y)\leq t$ implies $\forall i\in J_t^c: Y\prec Y_i$; i.e., $Y\in V_t$. Also, $f(Y)<f(Y_i)\leq t$ implies that $Y\in U_t$ a.s. So,
		\begin{eqnarray*}
			\myprob{Y\in V_t\setminus U_t, f(Y)\leq t} &=& \myprob{Y\not\in U_t, f(Y)\leq t}\\
			&\leq& \myprob{\forall i\in J_t: f(Y_i)\leq f(Y)\leq t}.
		\end{eqnarray*}
		If $\mu'(f(\cdot)\leq t)=0$, then the last term is zero. Otherwise, one can write
		\[
			\omidCond{Y\in V_t\setminus U_t}{f(Y)\leq t}\leq \omidCond{\forall i\in J_t: f(Y_i)\leq f(Y)}{f(Y)\leq t}.
		\]
		Conditionally on $f(Y)\leq t$, the sequence $(Y_i)_{i\in J_t}$, together with $Y$, are i.i.d. 
		In the next paragraph, we prove that the distribution of $f(Y_1)$ is diffuse. Assuming this, the last term is zero by exchangeability of $Y,Y_1,Y_2,\ldots$. One can similarly prove that $\myprob{Y\in V_t\setminus U_t, f(Y)> t}=0$. Thus, $\myprob{Y\in V_t\setminus U_t}=0$; i.e., $\mu'(V_t\setminus U_t)=0$, as desired.
		
		It remains to prove that the distribution of $f(Y_1)$ is diffuse. Assume that it has an atom at $m$. Let $p:=\mu'(f(\cdot)=m)>0$, $a:=\mu'(f(\cdot)<m)$ and $\lambda:=a+p$. Let $(Y'_i)_i$ be i.i.d. with distribution $\frac 1{p}\restrict{\mu'}{\{f(\cdot)=m\}}$. One has
		\begin{eqnarray*}
			m &=& \omidCond{f(Y_1)}{f(Y_1)=m} = \omid{\mu'(\cdot\prec Y'_1)} = \omid{\int \identity{\{x\prec Y'_1\}}d\mu'(x)}\\
			&=& \omid{\int \identity{\{x\prec Y'_1\}} \identity{\{f(x)<m\}}d\mu'(x)} + \omid{\int \identity{\{x\prec Y'_1\}} \identity{\{f(x)=m\}}d\mu'(x)}\\
			&=& a + p\myprob{Y'_2\prec Y'_1} = a+ \frac p 2,
		\end{eqnarray*}
		where the last equation holds by symmetry of $(Y'_1,Y'_2)$. Similarly,
		\begin{eqnarray*}
			m^2 &=&\omidCond{f(Y_1)^2}{f(Y_1)=m} = \omid{\mu'(\cdot\prec Y'_1)^2}\\
			&=& \omid{\int\int \identity{\{x\prec Y'_1,y\prec Y'_1\}}d\mu'(x)d\mu'(y)}.
		\end{eqnarray*}
		By splitting the last integral according to four terms according to whether $f(x)$ or $f(y)$ are equal to $m$, one obtains that
		\begin{eqnarray*}
			m^2 &=& a^2 + 2 ap\myprob{Y'_2\prec Y'_1} + p^2 \myprob{Y'_2\prec Y'_1, Y'_3\prec Y'_1}\\
			&=& a^2 + ap + \frac {p^2}{3} > (a+\frac p 2)^2 = m^2,
		\end{eqnarray*}
		a contradiction. So, it is proved that the distribution of $f(Y_1)$ is diffuse, and the proof of \Cref{claim:attain:=} is completed.
	\end{proof}
	
	From now on, fix a realization of $(Y_i)_i$ in the event $E'$ defined in \Cref{claim:attain:=}.	
	Let $\tau:=\tau(\Phi):=\inf\{t: V_t\cap\Phi\neq\emptyset\}$ with the convention $\tau =1$ if $V_{1}\cap\Phi=\emptyset$.
	
	\begin{claim}
		\label{claim:attain:final}
		One has $\mu(V_{\tau(\Phi)}\Delta Z(\Phi))=0$ and $\Phi\cap V_{\tau(\Phi)} = \Phi\cap Z(\Phi)$ a.s.
	\end{claim}
	\begin{proof}
		If $\Phi\cap Z(\emptyset)=\emptyset$, then \Cref{claim:attain:empty} shows that $Z(\Phi)=Z(\emptyset)$ and $\tau(\Phi)=1$. In this case, the claim holds since $V_1=Z(\emptyset)$ by definition. 
		
		If $\Phi\cap Z(\emptyset)\neq \emptyset$, as mentioned after \Cref{claim:attain:empty}, one almost surely has $Z(\Phi)=Z(x)\ni x$ for some $x\in \Phi$. Also, \Cref{claim:attain:zero} gives that $\mu(Z(\Phi)\setminus Z(\emptyset))=0$ a.s. Also, almost surely, $\prec$ gives a total order on the union of $\Phi\cap Z(\emptyset)$, $(Y_i)_i$ and almost all tuples $(x_1,\ldots,x_k)\in Z(\emptyset)^k$. Also, the proof of \Cref{claim:attain:=} shows that $(f(y): y\in \Phi\cap Z(\emptyset))$ are distinct a.s. Consider a sample of $\Phi$ that satisfies all these conditions. Since $\Phi\cap Z(x) = \{x\}$, the definition of $\prec$ implies that $x\prec x'$ for all $x\neq x'\in \Phi$, and thus, $f(x)<f(x')$. This implies that $\tau(\Phi)=f(x)$ and $\Phi\cap V_{f(x)} = \{x\}=\Phi\cap Z(\Phi)$. Also, since $x$ is comparable to almost all other points $y\in Z(\emptyset)$, the definition of $\prec$ implies that $y\prec x$ if and only if $y\in Z(x)$ for almost all $y$. Hence, $\mu(Z(x)\Delta \{\cdot\prec x\})=0$. Since $\{\cdot\prec x\}$ is a subset of $V_{f(x)}$ (modulo null sets) and these two sets have equal measures, one gets $\mu(Z(x)\Delta V_{f(x)})=0$. This proves the claim.
	\end{proof}
	
	Now, $(V_t)_{t\leq \tau}$ defines the first iteration of an exploration process up to time $t=1$ and starting from $C_0:=\emptyset$ (see \Cref{def:explorationnew}). There is no need to define the rest of the exploration since we stop at the first time a point is discovered, and we stop at $\tau=1$ if no point is discovered. \Cref{claim:attain:=,claim:attain:final} imply that this exploration process is $\mu$-continuous and satisfies the claim. So, the theorem is proved.
\end{proof}

\begin{proof}[Proof of \Cref{thm:attain2}]
	We only sketch the main steps of the proof.
	\Cref{thm:attain} constructs the first iteration of an exploration process up to time $t=1$.
	In the iterative algorithm in the definition of an exploration process (\Cref{def:explorationnew}), by a \textit{semi-stage} we mean a time interval of the form $(\tau_i+j-1, \min\{\tau_i+j, \tau_{i+1}\}]$ whenever $j\in\mathbb N$ and $\tau_i+j-1<\tau_{i+1}$. Inductively, we use $Z_1,\ldots, Z_n$ to construct the first $n$ semi-stages.
	Assume the first $n$ revealed points $\bar x:= (x_1,\ldots,x_n)$ are known (if the $i$'th iteration does not reveal any new point, let $x_i$ be undefined), which are assumed to satisfy $\forall i\leq n: \bar x\cap Z_i(\bar x)=\{x_1,\ldots,x_i\}$. Let $\bs N_{\bar x}$ be the set of $\varphi\in \bs N$ such that $\forall i\leq n: \varphi\cap Z_i(\varphi)=\{x_1,\ldots,x_i\}$. By applying \Cref{thm:attain} to the space $\mathbb X_{\bar x}:=\mathbb X\setminus Z_n(\bar x)$ and the measure $\restrict{\mu}{\mathbb X\setminus Z_n(\bar x)}$, one can construct an exploration process for $\varphi\in\bs N_{\bar x}$ that attains $Z_{n+1}$ at the first time a new point is revealed (or stopped at time $\tau_n+1$ if no new point is revealed). This constructs the next semi-stage. Hence, the whole exploration process is constructed inductively, which attains $(Z_i)_i$. It remains to show the measurability of this construction, which is omitted for brevity.
\end{proof}

\section{Problems}
\label{sec:problems}

The CTDTs used in \Cref{thm:ctdt} are assumed to be $\mathcal X_0$-valued by definition. Can this assumption be weakened to $\forall t: \mu(Z_t(\Phi))<\infty$? The difficulty is defining an exploration process without the condition $U_t\in X_0$, since infinitely many points may be discovered at once, or at arbitrarily small times. However, we expect this not to happen a.s. for the PPP. See \Cref{rem:empty}.

\Cref{thm:attain,thm:attain2} assume that $\mu$ is either diffuse or totally atomic. What are the analogous statements without these conditions? 
Also, the assumption in the definition of attainability (\Cref{def:attain}) is weaker than $\tilde Z_{\tau}(\Phi)=Z(\Phi)$ (see~\cite[Appendix~B]{LaPeYo23stopping}). Do the attainability theorems hold with this stronger condition? In addition, which point processes can be used in these theorems rather than the PPP? Are there simple attainability criteria for stopping sets with more than one revealed point?

In the proof of \Cref{thm:attain}, does $\prec$ lead to a total order as discussed before \Cref{claim:attain:include}?

\Cref{prop:conditionalPoisson} can be generalized to characterize when $(Y,\Phi)$ and $(Y,\Psi)$ have the same distribution, but needs some moment conditions. In the first version of this paper, the following condition was assumed (in the Poisson case): There exist $c_i>0$ for $i=1,2,\ldots$ such that $\forall i,n: \omid{\Phi^{(n)}(W_i^n)}\leq n!c_i^n$. What weaker moment conditions can be used in the general case?  For instance, is the Carleman-type condition $\forall i: \sum_n \omid{\Phi^{(n)}(W_i^n)}^{-1/n}=\infty$ sufficient?

\section*{Acknowledgements}
This work was initiated while working with Francois Baccelli on~\cite{BaKh26indistinguishability}. We thank him for his comments. AI is not used at all in the statements and proofs except after the first version was published on arXiv, where it was used for minor debugging and the following improvements: For removing the moment condition of \Cref{prop:conditionalPoisson}  (see \Cref{sec:problems}) and modifying its proof, and for finding \Cref{ex:counterexample}. ChatGPT-6 Astra and ChatGPT-6.1 Sol are used for this purpose, and for helping to write the introduction and the texts of the first version without touching the statements and proofs.

\bibliography{bib} 
\bibliographystyle{plain}

\end{document}